\documentclass[bj]{imsart}

\RequirePackage[OT1]{fontenc}
\RequirePackage{amsthm,amsmath,natbib}
\usepackage[utf8x]{inputenc}
\usepackage[english]{babel}

\usepackage{amsfonts}
\usepackage{amssymb}
\usepackage{color}
\usepackage{graphicx}
\usepackage{latexsym}
\RequirePackage[colorlinks,citecolor=blue,urlcolor=blue]{hyperref}

\arxiv{arXiv:0000.0000}

\startlocaldefs
\numberwithin{equation}{section}
\theoremstyle{plain}
\newtheorem{theorem}{Theorem}[section]
\theoremstyle{remark}
\newtheorem{remark}{Remark}[section]
\theoremstyle{lemma}
\newtheorem{lemma}{Lemma}[section]
\theoremstyle{corollary}
\newtheorem{corollary}{Corollary}[section]
\theoremstyle{definition}
\newtheorem{definition}{Definition}[section]
\endlocaldefs
\newcommand{\lb}{\left (}
\newcommand{\rb}{\right )}
\newcommand{\lbb}{\left [}
\newcommand{\rbb}{\right ]}
\newcommand{\labs}{\left |}
\newcommand{\rabs}{\right |}
\newcommand{\lbrb}[1]{\lb #1 \rb}
\newcommand{\lbbrbb}[1]{\lbb#1\rbb}

\newcommand{\lbbrb}[1]{\lbb#1\rb}
\newcommand{\lbrbb}[1]{\lb#1\rbb}

\newcommand{\lbcurly}{\left\{}
\newcommand{\rbcurly}{\right\}}

\newcommand{\abs}[1]{\labs#1\rabs}

\newcommand{\curly}[1]{\lbcurly#1\rbcurly}

\newcommand{\limi}[1]{\lim_{#1\to \infty}}

\newcommand{\liminfi}[1]{\liminf_{#1\to \infty}}

\newcommand{\Eb}{\mathbb{E}}

\newcommand{\Rb}{\mathbb{R}}
\newcommand{\Pb}{\mathbb{P}}

\newcommand{\Bc}{\mathcal{B}}

\newcommand{\Oc}{\mathcal{O}}

\newcommand{\Pbb}[1]{\mathbb{P}\lb #1\rb}
\newcommand{\Ebb}[1]{\Eb\lbb #1\rbb}

\newcommand{\LL}{L\'{e}vy }
\newcommand{\LLP}{L\'{e}vy process }
\newcommand{\LLPs}{L\'{e}vy processes }

\begin{document}

\begin{frontmatter}
\title{A Characterization of the Finiteness of Perpetual Integrals of L\'evy Processes}
\runtitle{Finiteness of Perpetual Integrals of L\'evy Processes}

\begin{aug}
\author{\fnms{Martin} \snm{Kolb}\thanksref{a,e1}\ead[label=e1,mark]{kolb@math.uni-paderborn.de}} \and 
\author{\fnms{Mladen} \snm{Savov}\thanksref{b,e2}\ead[label=e2,mark]{mladensavov@math.bas.bg}}
\address[a]{Institut f\"ur Mathematik, Universit\"at Paderborn, Warburger Str. 100,
	33098 Paderborn, Germany
\printead{e1}}

\address[b]{Institute of Mathematics and Informatics, Bulgarian Academy of Sciences, Akad. Georgi Bonchev street Block 8, Sofia 1113, Bulgaria
\printead{e2}}

\runauthor{M. Kolb et al.}

\affiliation{Some University and Another University}

\end{aug}

\begin{abstract}
\,\,We derive a criterium for the almost sure finiteness of perpetual integrals of \LL processes for a class of real functions including all continuous functions and for general one-dimensional L\'evy  processes that drifts to plus infinity. This generalizes previous work of D\"oring and Kyprianou, who considered L\'evy  processes having a local time, leaving the general case as an open problem. It turns out, that the criterium in the general situation simplifies significantly in the situation, where the process has a local time, but we also demonstrate that in general our criterium can not be reduced.  This answers an open problem posed in \cite{doring}.
\end{abstract}



\end{frontmatter}

\section{Introduction}

Let $\xi = (\xi_s)_{s \geq 0}$ be a L\'{e}vy process, that is $\xi$ is almost surely right-continuous stochastic process in $\Rb$, which possesses stationary and independent increments. In this paper we are interested in the finiteness of functionals of transient \LL processes, which escape to infinity, that is $\limi{s}\xi_s=\infty$. We provide a characterization of the almost sure finiteness of additive functionals of $\xi$, which are of the form 
\begin{equation}\label{i:perpetual}
\int_0^\infty f(\xi_s)\,ds.
\end{equation}
Questions of this type are of course a classical topic, which has been studied intensively for several classes of stochastic processes. The most recent paper, see \cite{K019}, deals with a large class of Feller processes and provides conditions  for almost sure infiniteness of functionals as in \eqref{i:perpetual} in terms of the symbol of the Feller processes, which can  be then specialized to \LLPs but yield less comprehensive results than ours. For classes of diffusion processes and in particular for Brownian motion results characterizing finiteness of perpetual integrals are well known and there exists a huge number of papers related to this problem. Therefore, we will only mention those approaches, which are most relevant for the present work. We will mainly rely on ideas of Charles Batty which have been developed in \cite{batty} for transient Brownian motion, that is a Brownian motion in $\Rb^d,d\geq 3$. Batty's methods  are  our main tools but due to the discontinuity of  \LLPs under consideration they have to be augmented and diversified sometimes in a non-trivial way. We would like to emphasize that Batty is mainly interested in a different purely analytic question related to Schr\"odinger operators. In the sequel we briefly sketch his results as these connections to spectral and potential theory do not seem to be well-known in the probability community. Let us consider Brownian motion $(B_t)_{t\geq 0}$ in $\mathbb{R}^d$ for $d\geq 3$ and  let $H:= -\frac{1}{2}\Delta + V$ ($0\leq V \in L^1_{loc}(\mathbb{R}^d)$) be the Schr\"odinger operator with potential $V$, then it is known that the semigroup $e^{-tH}$ can be represented as 
\begin{equation}\label{e:FK-formula}
e^{-tH}\psi(x) = \mathbb{E}_x\bigl[ e^{-\int_0^tV(B_s)\,ds}\psi(B_t)\bigr],
\end{equation}
where $0\leq \psi \in L^1_{loc}(\mathbb{R}^d)$. Batty is interested in finding necessary and sufficient criteria ensuring that 
\begin{displaymath}
\forall \psi \in L^1:\,\quad \|e^{-tH}\psi\|_{L^1(\mathbb{R}^d)} \rightarrow 0\quad \text{as $t\rightarrow \infty$},
\end{displaymath}
a property, which  is often called $L^1$-stability. Considering the Feynman-Kac-formula \eqref{e:FK-formula} it becomes clear, that there is a close connection between $L^1$-stability and the almost sure finiteness properties of the perpetual integral $\int_0^{\infty}V(B_s)\,ds$. Replacing Brownian motion $(B_t)_{t\geq 0}$ by a transient one-dimensional L\'evy process $\xi=(\xi_s)_{s\geq 0}$ leads to the the question of the present paper.  

In order to emphasize further connections to mathematical disciplines closely tied with probability we will sketch another approach to the problem above. It is possible to look at the problem from a more potential theoretic perspective by observing that the almost sure finiteness of the additive functional \eqref{i:perpetual} implies that the function $h$ defined by 
\begin{displaymath}
h(x)=\mathbb{E}_x\bigl[e^{-\int_0^{\infty}f(\xi_s)\,ds}\bigr]
\end{displaymath}
defines a non-trivial, non-negative and bounded function. If we denote by $\mathcal{L}$ the generator of the L\'evy process $\xi=(\xi_s)_{s\geq 0}$ then at least formally the function $h$ is harmonic for the operator $\mathcal{L}-f$, which is a more general version of the Schr\"odinger operator above. This connection has been successfully exploited in the context of Brownian motion in \cite{R08}.

The case of a transient L\'evy process instead of a Brownian motion does not seem to be fully understood until now and this note aims to contribute further insights. Motivated by previous work of D\"oring and Kyprianou covering the case of transient L\'evy processes with finite mean possessing a local time we will show how Batty's ideas lead to a characterization of the almost sure finiteness without the assumption of existence of a local time.  It is worth mentioning that these ideas are applicable in the context of processes with jumps since key stopping times related to additive functionals are announceable and therefore the process is almost surely continuous at these random moments.

\section{Notation and Main Results}
In order to fix the notation we recall that $\xi=(\xi_s)_{s \geq 0}$ denotes a one-dimensional L\'evy process which drifts towards plus infinity. We have already pointed out, that the questions have close connections to spectral and in particular to potential theory and it is therefore not surprising that notions from potential theory play a crucial role. Recall that the potential measure $U(dx)$ of a transient \LL process, defined as 
\begin{equation}\label{eq:potem}
\begin{split}
&U(dx) = \int\limits_{0}^{\infty}\mathbb{P}(\xi_s \in dx) \,ds,
\end{split}
\end{equation}
is a non-negative measure, which is finite on compact sets.  The measure $U$ can be interpreted as the expected occupation time measure.

%Since we can write $A_{E} = \bigcup_{t > 0}\bigcap_{s \geq t}\{\xi_s \in E\}$
Following \cite{batty} we introduce the following class of transient/recurrent sets. This class will be an essential ingredient in our characterization of the almost sure finiteness of perpetual integrals.
\begin{definition}\label{def:tra}
	Let $E$ be a Borel subset of $\mathbb{R}$ and let $x\in\Rb$. For brevity let $\Pb^x$ denote the law of $x+\xi$. The set $\mathbb{R} \setminus E$ is said to be $\Pb^x$-recurrent for the L\'{e}vy process if
	\begin{equation}\label{eq:trans}
	\mathbb{P}^x\left( \exists \tau\geq 0 : \xi_s \in E \; \forall s \geq \tau \right) = 0
	\end{equation}
	and is called  $\Pb^x$-transient if
	\begin{equation}
	\mathbb{P}^x\left( \exists \tau\geq 0 : \xi_s \in E \; \forall s \geq \tau \right) = 1.
	\end{equation}
\end{definition}
Observe that the set $\curly{\exists \tau\geq 0 : \xi_s \in E \; \forall s \geq \tau}$ belongs to the tail sigma algebra of the \LL process and therefore using \cite[Chapter 8, Exercise 10]{Dhavar} it has probability zero or one. 
The main insight leading to Definition \ref{def:tra} is that the behaviour of the function $f$ on a $\mathbb{P}^x$-transient set $\mathbb{R} \setminus E$ should not matter too much as far as the $\mathbb{P}^x$-almost sure finiteness of the perpetual integral $\int_0^{\infty}f(\xi_s)\,ds$ is concerned.  We will make a slight digression to elucidate an interesting question about triviality of the tail sigma algebra. If $\xi$ is a Brownian motion then if \eqref{eq:trans} holds for one $x$, it holds for all. 
This is due to the fact that in the case of Brownian motion the restriction of $\mathbb{P}_{\lambda}$ is independent of the initial distribution $\lambda$. 
This is not the case for general \LLPs as can easily be checked taking a counting Poisson process with different starting position. This independence property can be connected with the so called coupling property as given in \cite{T00}. The validity of the coupling property for \LLP has been investigated in detail \cite{S11}. Theorem 4.3 in \cite{S11} e.g. gives a necessary condition for the validity 
of the coupling property for \LL processes. 

Our main goal is to prove the following theorem which provides an analytic characterization of finiteness of functionals of \LL processes. The proof will be given in Section 4 below. In the subsequent Section 3 we will first compare our result with previous characterizations, which have been obtained for restricted classes of L\'evy processes.

\begin{theorem}
	\label{main_theorem}
	Let $f:\mathbb{R} \to \mathbb{R}^{+}$ be a non-negative, continuous or ultimately non-increasing, locally bounded function and $\xi = \{\xi_s\}_{s \geq 0}$ be a L\'{e}vy process such that  we have $\limi{s}\xi_s=\infty$. Let $x\in\Rb$. Then the following characterization is in place:
	\begin{equation}\label{eq:RHS}
	\begin{split}
	&\mathbb{P}\left( \int\limits_{0}^{\infty}{f(x+\xi_s)ds} < \infty \right) = 0 \\
	&\iff \int\limits_{E}f(x+y)U(dy) = \infty\,\text{ $\forall E\in\Bc\lbrb{\Rb}$ with  $\Rb\setminus E$ being $\Pb^x$-transient.}
	\end{split}
	\end{equation}
	Otherwise, if the right-hand integral is finite for some such $E$, then  $\int\limits_{0}^{\infty}f(x+\xi_s)ds<\infty$ almost surely.
\end{theorem}
\begin{remark}Let $f(x)=e^{-x}$ which is strictly decreasing on $\Rb$. Then since $(-\infty,0)$ is $\Pb$-transient and  $\int_{0}^{\infty}e^{-y}U(dy)$ is finite being the $U^{1}$ potential, we conclude that the following $\int_{0}^{\infty}e^{-\xi_s}ds<\infty$ holds. This elementary fact is well-known from the theory of exponential functionals of \LL processes, see \cite{PS18}, but illustrates the applicability of our result. We also emphasize that from Lemma \ref{lem:inf} and the proof of the main theorem for relation \eqref{eq:RHS} to hold it suffices to know that the set \[L^+_a(q):=\curly{x\in\Rb^+:\Pbb{\int_{0}^{\infty}f\lbrb{x+\xi_s}ds>a}\leq q}\] 
is closed. If this can be obtained by other means as in Remark \ref{rem:DK} the result would be valid.
\end{remark}
Observe that Theorem \ref{main_theorem} perfectly fits to the intuition that the behaviour of the function $f$ on a suitable class of negligible sets -- in our case transient sets -- should not contribute in the integral test. If the event $\mathbb{R} \setminus E$ is $\mathbb{P}$-transient then there is a finite (random) time $\tau$ such that after $\tau$ the path $\xi$ does not visit $\mathbb{R} \setminus E$ anymore and the behaviour of $f$ on $\mathbb{R} \setminus E$ can not be relevant for the finiteness of the perpetual integral. 
The next corollary extends Theorem \ref{main_theorem} in general and in particular gives conditions on the underlying \LLP  when it is valid for any non-negative, locally integrable, measurable function.
	\begin{corollary}\label{main_cor}
		Assume that the locally bounded
		$f=\limi{n}f_n$ is the pointwise limit on $L\subseteq \Rb$  of the non-decreasing sequence of real functions $\curly{f_n}_{n\geq 1}$ of the form
		\begin{equation*}
		\begin{split}
		&f_n=\sum_{j=1}^{k(n)} c^{(n)}_{j}\mathbb{I}_{(\alpha^{(n)}_j, \beta^{(n)}_j)}, 
		\end{split}
		\end{equation*}
		where $\mathbb{I}_{\cdot}$ stands for the indicator function of a set, $\alpha^{(n)}_j<\beta^{(n)}_j, \alpha^{(n)}_j,\beta^{(n)}_j\in \Rb, 1\leq j\leq k(n),$ and $c^{(n)}_{j}, n\geq 1,1\leq j\leq k(n),$ are non-negative real numbers. Let also $\Pbb{x+\xi_s\in L^c}=0,$ for all $s>0,x\in \Rb$.  Then, Theorem \ref{main_theorem} holds true for this $f$. In particular, for a given $\xi$,  it holds true for any non-negative, locally bounded, measurable function $f$ whenever  $\Pbb{x+\xi_s\in A}=0,$ for all $s>0,x\in \Rb$ and for any set $A$ of zero Lebesgue measure. The claim is also true even for non-negative, locally integrable functions if $\int_{0}^{t}f(x+\xi_s)ds<\infty$ for all $t>0$ and $x\in \Rb$.
	\end{corollary}
\begin{remark}\label{rem:genf} 
	 In the case when $\xi$ does not live on a lattice one can employ a mixture of indicators of open, semi-open and closed intervals in the form of  $f_n,n\geq 1$. Moreover, $\Pbb{x+\xi_s\in A}=0,$ for all $s>0,x\in \Rb$ and for any set $A$ of zero Lebesgue measure whenever the random variables $\xi_s$ are absolutely continuous for every $s>0$. 
\end{remark}
\begin{remark}\label{eq:exp}
	Since $\int_{0}^{\infty}f\lbrb{\cdot+\xi_s}ds$ is an additive functional then the Khas’minskii’s condition, see  \cite[p.120]{FP}, reads off as
		\begin{equation*}
		\begin{split}
		&J=\sup_{x\in\Rb}\int_{\Rb}f(x+y)U(dy)<\infty,
		\end{split}
		\end{equation*}
        where $\int_{\Rb}f(x+y)U(dy)$ equals $G_f\mathbf{1}(x)$ in the context of 	\cite[(4)]{FP}.
 In this case it holds true even that
	\begin{equation}\label{eq:K}
	\begin{split}
	&\Ebb{e^{\theta \int_{0}^\infty f(x+\xi_s)ds}}<\infty,\quad\text{ for $\theta<\frac{1}{J}$,}
	\end{split}
	\end{equation}
	see \cite[p.120]{FP}.
\end{remark}

\section{Connections to previous results}\label{sec:ex}
In this section we will clarify the relationship between our main result and those contained in the previous literature closest to our setting. More precisely, we will show, how previous results can be recovered from Theorem \ref{main_theorem}. Furthermore, we give an example demonstrating, that in general the integral test provided by Theorem \ref{main_theorem} can not be reduced to the integral test presented in \cite{doring} for a restricted class of transient L\'evy processes. This answers a question posed in \cite{doring}.
\subsection*{Relations to \cite{doring}:}
An illustrative example is the case discussed in \cite{doring}.  For convenience we recall the formulation of the main result of \cite{doring}. For this we assume that $\xi$ is a \LLP with $\mu=\Ebb{\xi_1}\in\lbrb{0,\infty}$ and that $\xi$ possesses a local time, see \cite[Chapter V]{bertoin} for more information on the existence of local time of \LL processes. Furthermore, we exclude the case of compound Poisson processes. Then, D\"oring and Kyprianou are able to prove the following interesting result: 

\begin{equation}\label{e:DK-Test}
\mathbb{P}\biggl(\int_0^{\infty}f(\xi_s)\,ds<\infty\biggr)=0\, \,\text{if and only if}\,\,\int^{\infty}f(s)\,ds = \infty,
\end{equation}
where in the aforementioned paper $\int^{\infty}$ stands to emphasize that the integrability is determined only in a neighbourhood of infinity as it is the case in our work too.

First we want to stress that the integral with respect to the Lebesgue measure to the right-hand side of \eqref{e:DK-Test} obviously can not be true in the case of compound Poisson processes and therefore was excluded from \eqref{e:DK-Test} in the paper \cite{doring}. If e.g. the process lives on the grid $\alpha \mathbb{N}$ with $\alpha >0$, then the non-negative continuous function 
\begin{displaymath}
f(x):= 1+\sin\biggl(\frac{3\pi}{2}+\frac{2\pi x}{\alpha}\biggr)
\end{displaymath}
will provide an easy counterexample. The Lebesgue measure is obviously the wrong measure in this situation, whereby $U$ naturally captures which sets are visited by the process.

In order to see the connection to Theorem \ref{main_theorem} in this situation we first observe that since the local time is assumed to  exist then single points are not essentially polar and hence the potential measure $U(dx)$ is absolutely continuous, whose canonical version of the density $u$ is bounded, see \cite[Chapter II, Theorem 16]{bertoin}. Then from \cite[Chapter II, Corollary 18]{bertoin} we have with 
\begin{equation*}
\sigma^x=\inf\curly{s\geq 0:\xi_s=x}
\end{equation*}
that
\begin{equation}\label{eq:hit}
\begin{split}
&\Pbb{\sigma^x<\infty}=\frac{u(x)}{u(0)}.
\end{split}
\end{equation}
For the first passage time 
 \begin{equation*}
	T^x=\inf\{s \geq 0:\xi_s\geq x\}
	\end{equation*} 
	let us now denote by  
\begin{equation*}
\Oc_x=\xi_{T^x}-x 
\end{equation*}
the overshoot of the process. Since $\mu\in\lbrb{0,\infty}$ it is known  that in distribution $\limi{x}\Oc_x=\Oc$ with $\Oc$ being a proper random variable with support in $\lbbrb{0,\infty}$, see \cite[Lemma 3]{bertoinS}. Moreover, either $\Oc$ has an atom at zero and/or there is an interval $\lbrb{0,a}$ on which it has non-increasing positive density, see again \cite[Lemma 3, $\rho_2$ in their notation]{bertoinS}. Therefore,
\begin{equation}\label{eq:stat}
\begin{split}
A&=\liminfi{x}\frac{u(x)}{u(0)}=\liminfi{x}\Pbb{\sigma^x<\infty}\\
&=\Pbb{\Oc=0}+\liminfi{x}\lbrb{ \int_{0+}^{\infty}\Pbb{\sigma^{-y}<\infty}\Pbb{\Oc_x\in dy}}\\
&=\Pbb{\Oc=0}+\liminfi{x}\lbrb{ \int_{0+}^{\infty}\frac{u(-y)}{u(0)}\Pbb{\Oc_x\in dy}}.
\end{split}
\end{equation} 
Clearly, when $\Pbb{\Oc=0}>0$ then $A>0$. The latter is also true if $\Pbb{\Oc=0}=0$. Indeed, note that first, the density of $\Oc$ is strictly positive and non-increasing on $\lbrb{0,a}$, and moreover from \cite[Chapter I, Proposition 12 and Proposition 11]{bertoin} we get that $u(-x)$ is excessive and therefore lower semi-continuous. Second, $U([-a,0])>0$ implies the existence of $y_0\in\lbrb{0,a}$ such that $h=u(-y_0)>0$ and from the lower semi-continuity $u(-y)>h/2$ on $y\in\lbrb{y_0-\eta,y_0+\eta}$ for some $\eta>0$. Finally, we can choose $\eta$ small enough that $(y_0-\eta,y_0+\eta)\subset\lbrb{0,a}$, which in turn gives from the Portmanteau's theorem that
\begin{equation*}
\begin{split}
\liminfi{x}\lbrb{ \int_{0+}^{\infty}\frac{u(-y)}{u(0)}\Pbb{\Oc_x\in dy}}&\geq \frac{h}{2}\liminfi{x}\Pbb{\Oc_x\in\lbrb{y_0-\eta,y_0+\eta}}\\
&\geq \frac{h}{2}\Pbb{\Oc\in\lbrb{y_0-\eta,y_0+\eta}}>0.
\end{split}
\end{equation*}
Thus, $A>0$ and the latter combined with \eqref{eq:stat} clearly yields that for some $\epsilon>0$ and $b>0$  
\begin{equation}\label{eq:new}
\begin{split}
&\Pbb{\sigma_x<\infty}\geq \epsilon,\quad \text{ for all $x\geq b$}.
\end{split}
\end{equation}
Set $C_x=\curly{\sigma_x<\infty}$. Then we can show that the structure of the $\Pb^x$-transient sets is very simple. We have that for any $x\in\Rb$, a set $E\subseteq \Rb$ is $\Pb^x$-transient if and only if there exists a finite $K$ such that $E\cap \lbrb{K,\infty}=\emptyset$.  This can be demonstrated in the following way. Let $E=\lbrb{x_n}_{n\geq 1}$ be any increasing sequence of positive numbers that converges to infinity. Without loss of generality we can assume that $x_1>b$. Then, $\Pbb{C_{x_{n}}}\geq \epsilon, \forall n\geq 1,$ and henceforth $\Pbb{C_{x_n}\text{ i.o.}}\geq \epsilon.$ This means that $\Pbb{\xi\text{ visits $E$ i.o.}}\geq \epsilon$ and from Definition \ref{def:tra} we conclude that $E$ cannot be transient as we have a zero-one law for it.
  The same argument shows that it is not $\Pb^x$-transient for any $x\in\Rb$. Therefore, any transient set is bounded away from plus infinity. The criterion of Theorem \ref{main_theorem} henceforth boils down to $\int_{0}^{\infty} f(x+y)u(y)dy<\infty$ but since above $A>0$, see \eqref{eq:stat}, and from \eqref{eq:hit} $u$ is bounded in the supremum norm, we see that this is equivalent to the condition in \cite[Theorem 1]{doring}, that is $\int_{0}^{\infty} f(y)dy<\infty$.

We point out that in this setting Corollary \ref{main_cor} holds for any non-negative, locally integrable function $f$. To briefly illustrate why, we recall the set $L^+_a(q)$ defined in \eqref{eq:L}, for every $q\in\lbbrb{0,1}$ and $a>0$, namely
\[L^+_a(q):=\curly{x\in\Rb^+:\Pbb{\int_{0}^{\infty}f\lbrb{x+\xi_s}ds>a}\leq q}.\]
Then in Remark \ref{rem:DK} below it is shown that for every $q\in\lbbrb{0,1}$ there is $a(q)>0$ such that for $a\geq a(q)$ the set $L^+_a(q)$ is closed for any non-negative, locally integrable function $f$. Thus, the requirement of the first item of Lemma \ref{lem:inf} is always satisfied and hence Corollary \ref{main_cor} holds for any non-negative, locally bounded function $f$. Let us assume that $f$ is locally integrable. Then the existence of local time easily yields that $\int_{0}^{t}f(x+\xi_s)ds<\infty$ almost surely for all $t>0, x\in \Rb$. Indeed, let $T^{[-a,a]}=\inf\curly{s\geq 0: \xi_s\notin\lbbrbb{-a,a}}$ and $L(t,x)$ be a non-decreasing in $t$ version of the local time. Then, for any $a>0$,
\begin{equation*}
\begin{split}
&	\int_{0}^{T^{[-a,a]}}f(x+\xi_s)ds\leq \int_{-a}^{a}f(x+y)L(\infty,y)dy	    
\end{split}
\end{equation*}
with 
\[\Ebb{\int_{-a}^{a}f(x+y)L(\infty,y)dy}=\int_{-a}^{a}f(x+y)u(y)dy<\infty\]
since $u$ is bounded in the supremum norm. Therefore, $\int_{0}^{T^{[-a,a]}}f(x+\xi_s)ds$ is almost surely finite for any $a>0$ and hence $\int_{0}^{t}f(x+\xi_s)ds<\infty$ almost surely for all $t>0$. Thus, Corollary \ref{main_cor} is again applicable.

When $0$ is regular for itself, see \cite[Chapter II]{bertoin} for the definition, then $u$ is continuous everywhere and one can take a limit in \eqref{eq:stat}. Thus, this special case is easier to deal with.

In \cite{doring} the authors speculate that the finiteness of $\int_{0}^{\infty} f(y)dy$ (note that they use $\int^{\infty} f(y)dy<\infty$, which due to local integrability is equivalent condition) might be necessary and sufficient for the almost sure finiteness of the corresponding additive functional 
\begin{displaymath}
\int\limits_{0}^{\infty}f(\xi_s)ds
\end{displaymath}
even without the assumption of $\xi$ having a local time. We provide an example when it is not the case. Take $\xi$ to be an increasing \LL process with no drift and jumps of infinite activity and size at most $1$.  Thus, $\xi$ has a finite mean and it does not possess a local time. Then, as above, in distribution $\limi{x}\Oc_x=\limi{x}\xi_{T^x}-x=\Oc$ but $\Oc$ has no mass at zero since the process does not creep up in this case, see \cite[Chapter III, Theorem 5]{bertoin}. Therefore, there is a sequence of numbers $\lbrb{x_n}_{n\geq 1}$ increasing to infinity and a sequence of positive numbers  $\curly{\epsilon_n}_{n\geq 1}$ decreasing to zero such that, for any $n\geq 1$,
\begin{equation}\label{eq:strip}
\begin{split}
&\Pbb{\Oc_{y}\in\lbrb{0,\epsilon_n}}\leq 2\Pbb{\Oc\in\lbrb{0,\epsilon_n}}\leq \frac{1}{n^2} \text{ for all  $y\geq x_n$.}
\end{split}
\end{equation}
Then, we construct the sequence of open non-intersecting intervals $\lbrb{\lbrb{\alpha_n,\beta_n}}_{n\geq 1}$ in the following way. We take $\alpha_1=x_1,\beta_1=x_1+\epsilon_1$. Next, we set $\alpha_2=\alpha_1+1+x_2,\beta_2=\alpha_2+\epsilon_2$. The reason for this choice comes from the fact that the overshoots of $\xi$ are at most of size $1$ since the jumps of the process are at most of size $1$. Thus, we ensure that $\xi_{T^{\alpha_1}}\leq \alpha_1+1$ and $\alpha_2-\xi_{T^{\alpha_1}}\geq x_2$. We do this recursively by setting $\alpha_{n+1}=\alpha_n+1+x_{n+1},\beta_{n+1}=\alpha_{n+1}+\epsilon_{n+1}$. Set $\Rb\setminus
E=\bigcup_{n=1}^\infty\lbrb{\alpha_n,\beta_n}$. Clearly, using successively the Markov property we arrive at
\begin{equation*}
\begin{split}
&\Pbb{\bigcup_{t\geq 0}\bigcap_{s>t}\curly{\xi_s\in 
		E}}=1-  \Pbb{\xi_{T^{\alpha_n}}\in\lbrb{\alpha_n,\beta_n} \text{ infinitely often }}\\
&\geq 1-\limi{n}\sum_{k=n}^{\infty}\Pbb{\xi_{T^{\alpha_n}}\in\lbrb{\alpha_n,\beta_n}}\geq 1-\limi{n}\sum_{k=n}^{\infty}\Pbb{\Oc_{\alpha_n}\leq \epsilon_n}\\
&\geq 1- \limi{n}\sum_{k=n}^{\infty}\frac{1}{k^2}=1.
\end{split}
\end{equation*}
Therefore, $\Rb\setminus
E$ is $\Pb^0$-transient. We take non-negative continuous functions $f_n,n\geq 1,$ each supported in $\lbrb{\alpha_n,\beta_n}$ respectively and such that $\int_{0}^{\infty} f_n(x)dx=1$. Define $f=\sum_{n=1}^{\infty }f_n$ and note that $\int_{0}^{\infty} f(x)dx=\infty$. On the other hand
\[\int_{E}f(y)U(dy)=0\]
since $f=0$ on $E$ by definition. From Theorem \ref{main_theorem} we get that 
\[\mathbb{P}\left( \int\limits_{0}^{\infty}f(\xi_s)ds < \infty \right) = 1,\]
which would contradict the integral test \eqref{e:DK-Test} of D\"oring and Kyprianou.
\subsection*{Relations to \cite{ErMa}:}
Let $f$ be ultimately non-increasing. In the case when $\Ebb{\xi_1}\in\lbrb{0,\infty}$ then the criterion of \cite[Theorem 1, (1.4)]{ErMa} has been  easily shown that to be equivalent to
\[\int_{0}^{\infty}f(\xi_s)ds<\infty \iff \int_{0}^{\infty} f(y)dy<\infty,\]
see \cite[p.82, below (3.2)]{ErMa}. 
From the Blackwell's theorem, see \cite[Chapter I, Theorem 21]{bertoin},  and the fact that $f$ is non-increasing, it is easily seen that
\begin{equation*}
\begin{split}
&\int_{0}^{\infty} f(y)dy<\infty\iff \int_{0}^{\infty} f(y)U(dy)<\infty.
\end{split}
\end{equation*}
Hence, from Theorem \ref{main_theorem} we deduce that
\begin{equation*}
\begin{split}
&\int_{0}^{\infty} f(y)dy<\infty\implies \int_{0}^{\infty}f(\xi_s)ds<\infty.
\end{split}
\end{equation*}
Assume that $\int_{0}^{\infty} f(y)dy=\infty$ and yet $\int_{0}^{\infty} f(\xi_s)ds<\infty$. Since $\xi$ spends a finite amount of time below any $l>0$ we can consider $f\mathbb{I}_{[l,\infty)}$ instead of $f$  with $l$ such that $f$ is non-increasing on $\lbbrb{l,\infty}$. Assume then that 
\[\int_{0}^{\infty} f\mathbb{I}_{[l,\infty)}(\xi_s)ds<\infty.\]
Then, from monotonicity,  for any $q\in\lbrb{0,1}$ there is $a_q>0$ such that
\[L^+_a(q)=\curly{x\geq 0: \Pbb{\int_{0}^{\infty} f\mathbb{I}_{[l,\infty)}(x+y)dy>a_q}\leq q}=\lbbrb{0,\infty}.\]
Lemma \ref{lemma3} applied with $f\mathbb{I}_{[l,\infty)}\mathbb{I}_{L^+_a(q)}=f\mathbb{I}_{[l,\infty)}$ and $x=0$ then shows that
\begin{equation*}
\begin{split}
&\Ebb{\int_{0}^{\infty}f\mathbb{I}_{[l,\infty)}(\xi_s)ds}<\infty
\end{split}
\end{equation*} and hence 
\[\int_{0}^{\infty}f\mathbb{I}_{[l,\infty)}(y)U(dy)<\infty.\] 
We  therefore arrive at contradiction with $\int_{0}^{\infty} f(y)dy=\infty$. Thus, we recover the criterion by Erickson and Maller when $\Ebb{\xi_1}\in\lbrb{0,\infty}$. 

We point out that \cite[Theorem 1]{ErMa} covers the case when $\limi{t}\xi_t=\infty$ almost surely and $\Ebb{\abs{\xi_1}}=\infty$. The last argument above implies in the same fashion that
\begin{equation*}
\begin{split}
&\int_{0}^{\infty} f(\xi_s)ds<\infty\implies \int_{0}^{\infty}f\mathbb{I}_{[l,\infty)}(y)U(dy)<\infty
\end{split}
\end{equation*}
for some $l>0$. Since $f$ is non-increasing then  $\limi{x}f(x)=0$ and hence by integration by parts
\begin{equation*}
\begin{split}
&\int_{0}^{\infty}f\mathbb{I}_{[l,\infty)}(y)U(dy)=\int_{l}^{\infty}U([0,y])df(y)<\infty.
\end{split}
\end{equation*}
The last condition is in line with \cite[Theorem 1]{ErMa}.

\section{Proof of Theorem \ref{main_theorem}}
Throughout the proofs we use $f$ to denote a non-negative, locally integrable function and we shall specify additional conditions if needed. We also use the notation
\begin{equation}\label{e:Ix}
I^x := \int\limits_{0}^{\infty}f(x + \xi_s)ds \in [0, \infty].
\end{equation}
We will drop the index $x$ altogether when $x = 0$. According to \cite[Lemma 5]{doring} we first recall that 
\begin{equation*}
\forall{x} \in \mathbb{R}: \, \mathbb{P}(I^x < \infty) \in \{0, 1\}.
\end{equation*} 
Let us define for every $a \in [0,\infty]$ the stopping time
\begin{equation}\label{e:treffer}
T^x_a := \inf{ \{ t > 0 : I^x_t = a  \} }\in\lbbrbb{0,\infty},
\end{equation}
where we have set
\begin{equation*}
I^x_t = \int\limits_{0}^{t}f(x+\xi_s)ds.
\end{equation*}
Observe that  $I^x_{\infty} = I^x$ according to \eqref{e:Ix}.
We start with the obvious lemma, for which we do not provide a proof.
\begin{lemma}
	Let $f$ be a non-negative, locally integrable function. For every $x\in\Rb$ and every realization of the underlying L\'{e}vy process $\xi$ the integral $I^x_t$ is a continuous function for $t\leq \inf\{s\geq0: I^x_s=\infty\}\in\lbrbb{0,\infty}$.
\end{lemma}
We recall the following definition of announceable stopping times (see e.g. Exercise 5.3 in \cite{bertoin}) as they play a key role in our study. 
\begin{definition}
	Let $\sigma$ be a stopping time. We call $\sigma$ an announceable stopping time if there exists a sequence of stopping times $\{\sigma_n\}_{n \in \mathbb{N}},$ such that  almost surely $\sigma_n < \sigma, \, \forall n\geq 1$ and $\sigma_n \to \sigma$, as $n \to \infty$.
\end{definition}

\noindent
Note that since until explosion $I^x_\cdot : [0, \infty) \to \mathbb{R}^+$ is a continuous function for every realization of the L\'{e}vy process $\xi$ and every $x\in\Rb$, we have the following result:

\begin{lemma}
	\label{announceable_stopping_time}
	For any $a > 0$ and any $x \in \Rb$ the stopping time $T^x_a$ -- defined in \eqref{e:treffer} -- is announceable. Moreover, the process $\xi$ is a.s. continuous at $T^x_a$ on $\{T^x_a < \infty\}$.
\end{lemma}

\begin{proof}
	The fact that $I^x_t$ is a continuous function in $t$ shows that $T^x_{a - \frac{1}{n}} < T^x_{a}, \, \forall n > \frac{1}{a},$ and $T^x_{a - \frac{1}{n}} \to \tilde{T}^x_{a}$ for some announceable stopping time $\tilde{T}^x_{a}$. From this it readily follows that $\tilde{T}^x_{a} \leq T^x_{a}$. However, since $f$ is non-negative, on $\{\tilde{T}^x_{a} < \infty \}\supseteq\{T^x_{a} < \infty \} $  we get that
	\[a = \lim_{n \to \infty}\left(a - \frac{1}{n}\right) = \lim_{n \to \infty} \int\limits_{0}^{T_{a - \frac{1}{n}}}{f(x+\xi_s) \, ds} = \int\limits_{0}^{\tilde{T}^x_{a}}{f(x+\xi_s) \, ds}.\]
	Therefore $\tilde{T}^x_{a}\geq T^x_{a}$ on $\{\tilde{T}^x_{a} < \infty\}$ and hence $\tilde{T}^x_{a}= T^x_{a}$. The claim about continuity of $\xi$ at  $T^x_{a}$ then follows from \cite[Prop. 1.7]{bertoin}.
\end{proof}
The next result gives an estimate on the first moment of $I^x_t$, which is inspired from \cite[Lemma 2.2.]{batty}. Note that it is crucial in the proof that the involved stopping times are announceable. For this purpose we introduce, for $t\in\lbrbb{0,\infty},a>0,$ the quantity
\begin{equation}\label{eq:}
\begin{split}
&\beta_{f, a,t}:=\inf_{y \in supp(f)}\Pb\lbrb{I^y_t\leq a}\in\lbbrbb{0,1}
\end{split}
\end{equation}
and we have the following result. 
%\begin{lemma}
%\label{lemma2}
%Let $x \in \mathbb{R}$ and $\mathbb{P}^{x}(I < \infty) = 1$. Then, for any $a > 0$
%$$\frac{1}{a} \geq \beta_{f, a} \mathbb{E}^{x}\left[I\right] = \beta_{f, a}\mathbb{E}\left[ I^x \right],$$
%where $\beta_{f, a} = \operatorname*{inf}_{y \in \operatorname*{supp}{f}}\mathbb{P}^{y}\left[I \leq a \right]$.
%\end{lemma}
%
%\begin{proof}
%First we note that since $I_t$ is a continuous function in $t$ for every realization of the L\'{e}vy process $\xi$ and the process itself is a.s. continuous at the (announceable) stopping times $T_a$ then for every $n \in \mathbb{N}$ on $\{T_{an} < \infty\}$ almost surely $\xi_{T_{an}} \in supp(f)$. \\ \\
%Clearly, for $a > 0$ we have that 
%\begin{multline*}
%1 = \sum\limits_{n \geq 0}{\mathbb{P}^{x}(I \in (na, (n+1)a])} = \sum\limits_{n \geq 0}{\mathbb{P}^{x}(T_{na} < \infty, \, I \leq (n + 1)a)} = \\
%\sum\limits_{n \geq 0}{\mathbb{P}^{x}(T_{na} < \infty, \, \mathbb{P}^{\xi_{T_{na}}}(I \leq a))}
%\end{multline*}
%where we have used the strong Markov Property of $\xi$ at $T_{na}$. \\
%From \ref{announceable_stopping_time} we have that $T_{na}$ is an announceable stopping time and on $\{ T_{na} < \infty \}$ that $\xi$ is continuous at $T_{na}$. Therefore $\mathbb{P}^x(\xi_{T_{na}} \in supp(f)) = 1$ and we can deduce that
%$$1 \geq \beta_{f, a} \sum\limits_{n \geq 0}{\mathbb{P}^x(T_{na} < \infty)} = \beta_{f, a}{\sum\limits_{n \geq 0}{\mathbb{P}^x(I > na)}} \geq \frac{\beta_{f, a}}{a}\mathbb{E}[I^x]$$
%which concludes the proof of the lemma.
%\end{proof}

\begin{lemma} \label{lemma2}
	Let $x \in \mathbb{R}$. Then, for any $a > 0,t\in\lbrbb{0,\infty}$
	\begin{equation}\label{eq:beta}
	\begin{split}
	&\beta_{f, a,t} \mathbb{E} \left[ I^x_t \right] \leq a \mathbb{P}\left[ I^x_t < \infty \right]
	\end{split}
	\end{equation}
	with the convention that $0\times \infty=0$ to the left-hand side of the claim above.
\end{lemma}
\begin{remark}
	Observe that Lemma \ref{lemma2} in particular demonstrates that the assumptions $\beta_{f,a,t} > 0$ and $\mathbb{P}\left[ I^x_t < \infty \right]>0$ actually imply that the random variable $I^x_t$ even has a finite first moment.
\end{remark}
\begin{proof}
	First, we note that since $I^x_t$ is a continuous function in $t$ for every realization of the L\'{e}vy process $\xi$ and the process itself is a.s. continuous at the (announceable) stopping times $T^x_a$, see Lemma \ref{announceable_stopping_time},  then for every $n \in \mathbb{N}$ on $\{T^x_{an} < \infty\}$ almost surely $x+\xi_{T^x_{an}} \in supp(f)$. 
	Clearly, for $a > 0$, we have that
	\begin{equation*}
	\begin{split}
	\mathbb{P} \left[ I^x_t < \infty \right] &\geq \sum\limits_{n \geq 0}{\mathbb{P}(I^x_t \in (na, (n+1)a])} \\
	&= \sum\limits_{n \geq 0}{\mathbb{P}(T^{x}_{na}\leq  t, \, I^x_t \leq (n + 1)a)} \\
	&= \sum\limits_{n \geq 0} \Pbb{ T^x_{na} \leq t, \int_{T^x_{na}}^{t}f(x+\xi_s)ds \leq a}\\
	& \geq \sum\limits_{n \geq 0} \Pbb{  T^x_{na} \leq  t, \int_{T^x_{na}}^{T^x_{na} + t}f(x+\xi_s)ds \leq a } \\
	&=\sum\limits_{n \geq 0} \mathbb{E}\left[ \mathbb{I}_{\curly{T^x_{na} \leq  t}} \mathbb{P}\left[ I^{x+{\xi_{T^x_{na}}}}_t \leq a \right] \right] \geq \beta_{f, a,t} \sum\limits_{n \geq 0} \mathbb{P} \left[ T^x_{na} \leq  t \right]\\
	& = \beta_{f, a,t}\sum\limits_{n \geq 0}\mathbb{P}\left[ I^x_t \geq na \right] 
	\geq\frac{\beta_{f, a,t}}{a} \mathbb{E}[I^x_t],
	\end{split}
	\end{equation*}
	which ultimately leads to the stated inequality.
\end{proof}
\noindent
For any $a > 0, \, x \in \mathbb{R}$, we now introduce the functions
\begin{equation}\label{eq:G}
G_{a}(x) := \mathbb{P}(I^{x} > a)
\end{equation}
and the sets
\begin{equation}\label{eq:L}
L_{a}(q) := \{x \in \mathbb{R} : G_{a}(x) \leq q\},\quad L^+_{a}(q) := \{x \in \Rb^+ : G_{a}(x) \leq q\},
\end{equation}
where $q \in [0, 1]$ and $\Rb^+=\lbbrb{0,\infty}$. For the remaining parts of the proof of Theorem \ref{main_theorem} it will be essential that for $q\in\lbbrb{0,1}$ and for $a>0$ the sets $L_a(q)$ or/and $L^+_{a}(q)$ are closed. Here further properties of the function $f$ seem to be required. Our next result provides sufficient conditions for the set $L_a(q)$ and $L^+_{a}(q)$ to be closed.
\begin{lemma}\label{lem:L}
	Suppose the locally integrable and non-negative function $f$ is the pointwise limit on a Borel set $L\subseteq \Rb$, that is $f=\limi{n}f_n$,  such that 
	\begin{equation*}
	\forall s>0,x\in\Rb:\, \,\Pbb{x+\xi_s\in L^c}=0
	\end{equation*}
	and where $\{ f_n \}_{n = 1}^{\infty}$ is a non-decreasing sequence of non-negative simple functions  of the type 
		\begin{equation}\label{eq:fn}
	\begin{split}
	&	f_n=\sum\limits_{i = 0}^{k(n)}{c^{(n)}_i \mathbb{I}_{(\alpha^{(n)}_i, \beta^{(n)}_i)}},	    
	\end{split}
	\end{equation}
	with $c^{(n)}_i\geq 0$ and $\{ (\alpha^{(n)}_i, \beta^{(n)}_i) \}_{i = 1}^{k(n)}$ are mutually disjoint. Then the sets $L_a(q)$ and $L^+_a(q)$  ($q \in [0, 1)$) are closed.
\end{lemma}
\begin{proof}
	In the subsequent proof we will add an additional superscript $f$ to our notation in order to make the dependence on properties of $f$ explicit.
	
	Let us first assume that  $f = \mathbb{I}_{(\alpha, \beta)},\alpha<\beta$ and let $\{ x_i \}_{i = 1}^{\infty}$ be a sequence in $L_a^f(q)$ converging to $x_0$. Applying Fatou's lemma we get that
	\begin{equation*}
	\begin{split}
	&\liminf_{i \to \infty}I^{x_i, f}=\liminf_{i \to \infty}{\int_{0}^{\infty}{\mathbb{I}_{ (\alpha, \beta)}(x_i + \xi_s )ds}} \geq \int_{0}^{\infty}{\liminf_{i \to \infty}\mathbb{I}_{ (\alpha, \beta)}(x_i + \xi_s )}ds\geq I^{x_0, f}.
	\end{split}
	\end{equation*}
	From the definition of $L_a^f(q)$ and from another application of Fatou's lemma, see \eqref{eq:L}, this leads to
	\begin{equation*}
	\begin{split}
	&q \geq \liminf_{i \to \infty}{\mathbb{P}\left[ I^{x_i,f}>a \right]}  \geq \mathbb{P}(I^{x_0, f}>a)=G^f_a(x_0),
	\end{split}
	\end{equation*}
	which proves that $x_0 \in L_a^f(q)$ hence $L_a^f(q)$ is closed.  
	
	In the second step assume that $f$ is a simple non-negative function of the form $f = \sum\limits_{i = 0}^{n}{c_i \mathbb{I}_{(\alpha_i, \beta_i)}}$, where we assume that the sets $\{ (\alpha_i, \beta_i) \}_{i = 1}^n$ are mutually disjoint. Then the very same argument is applicable and thus $L_a^f(q)$ is again closed. 
	
	Finally, let $f$ be any non-negative, locally integrable function such that on $L\subseteq \Rb$ it holds that 
	$\limi{n}f_n=f$, where  the non-negative functions $\{ f_n \}_{n = 1}^{\infty}$ are defined in \eqref{eq:fn} and $\Pbb{x+\xi_s\in L^c}=0$ for all $x\in\Rb$ and $s>0$.
	  We prove that
	\begin{equation}\label{eq:LL}
	\begin{split}
	&L_a^f(q) = \bigcap\limits_{n = 1}^{\infty}{L_a^{f_n}(q)}
	\end{split}
	\end{equation}
	and therefore the set  $L_a^f(q)$ is closed. Since we approximate $f$ with a non-decreasing sequence of functions it is clear that $I^{x,f_n}\leq I^{x,f}$ for any $n\geq 1$ and therefore
	\begin{equation*}
	\begin{split}
	&L_a^f(q) \subseteq \bigcap\limits_{n = 1}^{\infty}{L_a^{f_n}(q)}.
	\end{split}
	\end{equation*}
	Take a point $x \in \bigcap\limits_{n = 1}^{\infty}{L_a^{f_n}(q)}$. Hence, for every $n \in \mathbb{N}$ it holds that $\mathbb{P}(I^{x,f_n} > a) \leq q$.  Define the sets
	\[A_n = \{I^{x,f_n} > a\}.\]
	Due to the property of the approximating  sequence $\curly{f_n}_{n\geq 1}$ the sequence of sets $\curly{A_n}_{n\geq 1}$
	is increasing. Also, we have that  almost surely
	\begin{equation*}
	\int_{0}^{\infty} f(x+\xi_s)\mathbb{I}_{L^c}(x+\xi_s)ds=\sup_{n\geq 1}\int_{0}^{\infty} f_n(x+\xi_s)\mathbb{I}_{L^c}(x+\xi_s)ds=0
	\end{equation*}
	since by assumption
	\begin{displaymath}
	\int_{0}^{\infty} \Pbb{x+\xi_s\in L^c}ds=0. 
	\end{displaymath}
	Therefore, we conclude that $I^{x,f}=I^{x,f\mathbb{I}_{L}}, I^{x,f_n}=I^{x,f_n\mathbb{I}_{L}}$ simultaneously almost surely and thus using the monotone convergence theorem we finally arrive at 
	\[\Pbb{\bigcup\limits_{n = 1}^{\infty}{A_n}} =\Pb\{I^{x,f} > a\}.\]
	%Therefore from the monotone convergence theorem 
	%we conclude that
	%\[\lim_{n \to \infty}\Pbb{A_n}=\lim_{n \to \infty} \mathbb{P}(I^{x,f_n} > a) = \mathbb{P}(I^{x,f} > a)=\Pbb{\bigcup\limits_{n = 1}^{\infty}{A_n}}.\]
	Hence, $x\in L_a^f(q)$ and \eqref{eq:LL} holds true.
	Summarizing we have shown, that under the assumptions above the sets $L^f_a(q)$ is closed. When $L^f_a(q)$ is closed clearly $L^{f,+}_a(q)$ is closed too. This settles the claim of this lemma. 
\end{proof}
\begin{corollary}\label{rem:classf}
	Assume that $f$ has support $[l,\infty)$ for some $l \in \mathbb{R}$, and that $f|_{[l,\infty)}$ is either continuous or non-increasing. Then the sets $L^f_a(q)$ and $L^{f,+}_a(q)$  ($q \in [0, 1)$) are closed.
\end{corollary}
\begin{proof}
	We first note that the continuous functions fall into domain of Lemma \ref{lem:L}. However, in this case there is a more direct argument to prove closedness of since $L^f_a(q)$ and $L^{f,+}_a(q)$ from Fatou's lemma \[\liminf_{i \to \infty}I^{x_i,f}\geq I^{x,f}\] provided $\lim_{i \to \infty}x_i=x$.  Also, the ultimately non-increasing functions are captured by Lemma \ref{lem:L}. Let $f$ be non-increasing on $supp f=[l,\infty)$. In this case if $x\in L^f_a(q), x>l,$ then $\lbbrb{x,\infty}\subset L_a(q)$. Then, the lemmata below can be applied with $L^{f,+}_a(q)\cap \lbbrb{x,\infty}$, see Remarks \ref{rem:closed1} and \ref{rem:closed}.
\end{proof}
\begin{remark}\label{rem:nonlattice}
	It is worth noting that except in the case when the process $\xi$ lives on a lattice we can replace the open intervals  $(\alpha, \beta)$ used in the definition of $f_n, n\geq1,$ see \eqref{eq:fn} Lemma \ref{lem:L}, with closed ones, that is $[\alpha, \beta]$. Indeed the Lebesgue measure of $\{ \xi_s = \alpha - x_0\}$ or $\{ \xi_s = \beta - x_0 \}$ is $0$, which can be shown as follows: 
	\[0 = \mathbb{E}\left[\int_{0}^{\infty}\mathbb{I}_{\xi_s = \alpha - x_0}ds\right] = \int_{0}^{\infty}\mathbb{P}\left(\xi_s = \alpha - x_0\right)ds.\] 
\end{remark}
\begin{remark}\label{rem:DK}
	For any $q\in\lbrb{0,1}$ and all $a>0$ large enough we can explicitly compute the set $L_a(q)\cap\lbbrb{K,\infty}$ under the assumptions of  \cite{doring}, see Section \ref{sec:ex}, and $K>0$ depending only on $\xi$. We do this for every $q \in (0, 1)$ by allowing $a$ to depend on $q$. Note that we have shown that $A$ as defined in \eqref{eq:stat} is positive and therefore $C=u(0)/\inf_{x\geq K} u(x)>0$ for some $K>0$ and where $u$ is the density of the potential measure defined in \eqref{eq:potem}, see \cite[Corollary 2.36]{bertoin}. With these $K,C$ we prove that for every $q \in (0, 1)$ there is $a_q>0$ such that, for all $a\geq a_q$,
	\begin{equation}\label{eq:LK}
	\begin{split}
	&\operatorname*{sup}_{x \geq K} \mathbb{P}(I^x > a) \leq \mathbb{P}(I > a)\leq  \mathbb{P}(I > a_q)< C^{-1}q' =: q.
	\end{split}
	\end{equation}
	Let us sketch the proof. We note that with $x\in\Rb, x\neq 0$,
	\begin{equation}\label{eq:1}
	\begin{split}
	&I\geq \mathbb{I}_{\curly{\sigma^x<\infty}}\int_{\sigma^x}^{\infty}f(\xi_s)ds= \mathbb{I}_{\curly{\sigma^x<\infty}}\int_{0}^{\infty}f(x+\tilde{\xi}_s)ds,\\
	&I_x\geq \mathbb{I}_{\curly{\sigma^{-x}<\infty}}\int_{0}^{\infty}f(\tilde{\xi}_s)ds
	\end{split}
	\end{equation}
	where $\sigma^x=\inf\curly{s\geq 0:\xi_s=x}$ and $\tilde{\xi}_s=\xi_{s+\sigma^{\pm x}}-\xi_{\sigma^{\pm x}},s\geq 0$. Observe that the zero-one law, see \cite[Lemma 5]{doring}, implies that $I=\infty$ almost surely is equivalent to $I^x=\infty$ almost surely for any $x\in\Rb$. 
	 We note that when $x\geq K$ we have that
	\[\Pbb{\sigma^x<\infty}=\frac{u(x)}{u(0)}>C^{-1}>0.\]
	In the case $I=\infty$ almost surely we see that $L_a(q)\cap \lbbrb{K,\infty}=\emptyset$ for any $a>0$ and $q<1$. So assume that $I<\infty$ almost surely. The first relation \eqref{eq:1} then implies that for any $a>0$ and any $x\geq 0$
	\begin{equation*}
	\begin{split}
	&\Pbb{I>a}\geq \frac{u(x)}{u(0)}\Pbb{I^x>a}
	\end{split}
	\end{equation*}
	and when $x\geq K$
	\[C\Pbb{I>a}\geq \sup_{x\geq K}\Pbb{I^x>a},\]
	where recall that $C=u(0)/\inf_{x\geq 0} u(x)>0$. Since $\limi{a}\Pbb{I>a}=0$ for every $q\in\lbrb{0,1}$ we can choose $a_q$ such that
	\[ \sup_{x\geq K}\mathbb{P}(I^x > a_q) \leq q,\]
	which proves that for every $a \geq a_q$ we have that $L^+_a(q)\cap\lbbrb{K,\infty} = [K, \infty)$. 
\end{remark}

Lemma \ref{lemma2} allows to deduce the following Lemma.

\begin{lemma}
	\label{lemma3}
	If $L^+_{a}(q)$ is a closed set for some $a > 0$ and some $q \in [0, 1)$ then we have that 
	\begin{equation*}
	\mathbb{E}\left[ \int\limits_{0}^{\infty}(f\mathbb{I}_{L^+_{a}(q)})(x + \xi_s)ds \right] < \infty
	\end{equation*}
	for every $x \in \mathbb{R}$.
\end{lemma}
\begin{remark}\label{rem:closed1}
	Note that the statement of the lemma is also valid if $L^+_a(q)$ is replaced by $L^+_{a}(q)\cap [K,\infty)$, for some $K>0$, and the latter is assumed closed.
\end{remark}
\begin{proof}
	We first note that  
	\begin{equation*}
	\operatorname*{supp}(f\mathbb{I}_{L^+_{a}(q)}) = \operatorname*{supp}(f) \cap L^+_{a}(q)
	\end{equation*}
	if $L^+_{a}(q)$ is a closed set. \\
	We will distinguish the following two cases. First, let $x \in L^+_a(q)\subseteq L_{a}(q)$.  \\
	In this case we have that $G_a(x) = \mathbb{P}(I^{x} > a) \leq q < 1$ and using the $0-1$ law of \cite[Lemma 5]{bertoin} we can conclude that $I^x < \infty$ almost surely. Moreover, almost surely
	\[\int\limits_{0}^{\infty}(f\mathbb{I}_{L^+_{a}(q)})(x + \xi_s)ds\leq \int\limits_{0}^{\infty}(f\mathbb{I}_{L_{a}(q)})(x + \xi_s)ds \leq \int\limits_{0}^{\infty}f(x + \xi_s)ds = I^{x} < \infty\]
	and thus 
	\begin{equation*}
	\mathbb{P}\left( \int\limits_{0}^{\infty}(f\mathbb{I}_{L^+_{a}(q)})(x + \xi_s)ds  < \infty \right) = \mathbb{P}\left( \int\limits_{0}^{\infty}(f\mathbb{I}_{L_{a}(q)})(x + \xi_s)ds  < \infty \right) = 1.
	\end{equation*}
	Now let us consider the second case, i.e. the case $ x \in \Rb\setminus L^+_a(q)$. We define the stopping time
	\begin{equation*}
	\rho^x = \operatorname*{inf}\lbrace s \geq 0 : \xi_s + x \in L^+_a(q)\rbrace
	\end{equation*}
	and distinguish two separate cases for $\rho^x$. First, let us work on the event that $\rho^x = \infty$.
	Then on this event we have  $(f\mathbb{I}_{L^+_{a}(q)})(x + \xi_s) = 0$ for all $s$ large enough as $\xi$ is transient and therefore we conclude that on $\lbrace \rho^x = \infty\rbrace$ 
	\begin{equation*}
	\int\limits_{0}^{\infty} (f\mathbb{I}_{L^+_{a}(q)})(x + \xi_s)\,ds < \infty. 
	\end{equation*}
	Second, let us assume that $\rho^x < \infty$. Then on this event we have that
	\begin{equation*}
	\begin{split}
	&\int\limits_{0}^{\infty}(f\mathbb{I}_{L^+_{a}(q)})(x + \xi_s)ds = \int\limits_{\rho^x}^{\infty}(f\mathbb{I}_{L^+_{a}(q)})(x + \xi_s)ds \\
	&= \int\limits_{0}^{\infty}(f\mathbb{I}_{L^+_{a}(q)})(x + \xi_{\rho^x} + \tilde{\xi}_s)ds,
	\end{split}
	\end{equation*}
	where $\tilde{\xi}=\lbrb{\tilde{\xi}_s}_{s\geq 0}=\lbrb{\xi_{\rho^x+s}-\xi_{\rho^x}}_{s\geq 0}$.
	From the definition of $\rho^x$ and the closedness of $L^+_a(q)$ we can conclude that 
	\[x + \xi_{\rho^x} \in L^+_a(q)\cup \partial L^+_a(q) = L^+_a(q).\]
	Hence, from the first case of the proof we conclude that \[\Pbb{\int\limits_{0}^{\infty}{(f\mathbb{I}_{L^+_{a}(q)})(x + \xi_s) \, ds}  < \infty}=1.\]
	Therefore, from Lemma \ref{lemma2} applied with $t=\infty$ we conclude that
	\begin{equation*}
	\begin{split}
	&a=a \mathbb{P}\left[\int\limits_{0}^{\infty}(f\mathbb{I}_{L^+_{a}(q)})(x + \xi_s) \, ds<\infty\right] \geq \beta_{f\mathbb{I}_{L^+_{a}(q)}, a,\infty} \mathbb{E}\left[ \int\limits_{0}^{\infty}{(f\mathbb{I}_{L^+_{a}(q)})(x + \xi_s) \, ds} \right].
	\end{split}
	\end{equation*}
	However, 
	\begin{equation*}
	\begin{split}
	\beta_{f\mathbb{I}_{L^+_{a}(q)}, a,\infty} &= \operatorname*{inf}_{y \in suppf \cap L^+_{a}(q)} \mathbb{P}\left(\int\limits_{0}^{\infty}{(f\mathbb{I}_{L^+_{a}(q)})(x + \xi_s) \, ds} \leq a \right)\\
	&= \operatorname*{inf}_{y \in suppf \cap L^+_{a}(q)} \mathbb{P}\left(\int\limits_{0}^{\infty}{(f\mathbb{I}_{L^+_{a}(q)})(x +\xi_{\rho^x}+ \tilde{\xi}_{s}) \, ds} \leq a \right)\\ 
	&\geq \operatorname*{inf}_{y \in L^+_{a}(q)} \mathbb{P}(I^{y} \leq a) \geq 1 - q > 0.
	\end{split}
	\end{equation*}
	This concludes the proof of the lemma, when $L^+_a(q)$ is not the empty set. In case it is the whole integral is trivially $0$ and therefore the expectation of it is also zero. 
\end{proof}

\noindent

Recall Definition \ref{def:tra}. Since for every continuous function $f$ or any $f$ that can be approximated as in Remark \ref{rem:genf},  we know from Lemma \ref{lem:L} and Corollary \ref{rem:classf} that  $L_{a}(q), L^+_{a}(q)$ are closed for all $a > 0, \, q \in [0, 1)$ then the following result yields the backward claim of Theorem \ref{main_theorem} in this case. Also, if $f$ is any non-increasing function then according to the proof of Corollary \ref{rem:classf}, $L^+_{a}(q)$ is either empty or contains an interval of the type $[x,\infty),x>0$, and the next lemma is valid, see Remark \ref{rem:closed}. However, we note the weaker assumption it contains regarding the closedness of $L^+_{a}(q)$. 
\begin{lemma}\label{lem:inf}
	Let $x\in\Rb$ be fixed and assume that 
	\begin{itemize}
		\item\label{it:1} either $L^+_{a}(q)$ is closed for some fixed $q\in\lbrb{0,1}$ and across some sequence $\curly{a_j}_{j\geq 1}$ such that $a_j\to\infty$, or for a sequence of pairs $\curly{\lbrb{q_j,a_j}}$ such that $\limi{j}q_j=1$ and $\limi{j}a_j=\infty$;
		\item $\int\limits_{E}{f(x+y) \, U(dy)} = \infty$ for all sets $E$ such that $\mathbb{R} \setminus E$ is $\Pb^x$-transient for the L\'{e}vy process $\xi$.
	\end{itemize}
	Then $\mathbb{P}(I^x = \infty) = 1$.
\end{lemma}
\begin{remark}\label{rem:closed}
	Note that the statement of the lemma is valid even if $L^+_{a}(q)\cap [K,\infty)$ satisfies the assumptions of the first item for some $K>0$.
\end{remark}
\begin{proof}
	Choose $a>0,q\in\lbrb{0,1}$ such that $L^+_{a}(q)$ is closed. From Lemma \ref{lemma3} and the assumptions therein we have that for all $v \in \mathbb{R}$
	\[\infty > \mathbb{E}\left[ \int_{0}^{\infty}f\mathbb{I}_{L^+_a(q)}(v + \xi_s)ds \right] = \int_{L^+_a(q)}{f(v + y) \, U(dy)},
	\]
	where we recall that the last identity is in fact the definition of the potential measure for transient \LL  processes.
	Choosing $v=x $ in the relation above we get that $\Rb\setminus L^+_a(q)$ is not $\Pb^x$-transient for $\xi$ and since $\xi$ is transient then even $(0,\infty)\setminus L^+_a(q)$ is not $\Pb^x$-transient for $\xi$. Under the assumptions we have that $M = (0,\infty) \setminus L^+_a(q)$ is an open set and therefore, for  $a > 0, \, q \in [0, 1)$ for which  $L^+_{a}(q)$ is closed, $M$ a union of an increasing sequence of compact sets $K_n$, that is
	\[M = \cup_{n \geq 0}K_n.\]
	Set $H_n = \inf_{s \geq 0}\{ x+\xi_s \in K_n \}$. Since $M$ is not $\Pb^x$-transient for $\xi$ we get that $M$ must be visited by $x+\xi$ almost surely in a finite amount of time and therefore 
	\[\lim\limits_{n \to \infty}{\mathbb{P}^x(H_n < \infty)} = 1.\]
	Then the Markov property at $H_n$ and $K_n\subseteq M$ yield the following sequence of inequalities
	\begin{equation*}
	\begin{split}
	&\mathbb{P}(I^x > a) \geq \mathbb{E}^x \left[ 1_{\{ H_n < \infty \}} \mathbb{E}\left[ 1_{\{ I^{x+\xi_{H_n}}  > a \}} \right] \right] \geq \mathbb{P}^x(H_n < \infty) \inf_{y \in K_n}{\mathbb{P}(I^y > a)} \\
	&\geq q\,\mathbb{P}^x(H_n < \infty).
	\end{split}
	\end{equation*}
	If we let $n \to \infty$ we arrive at 
	\begin{equation}\label{eq:inter}
	\begin{split}
	&\mathbb{P}(I^x > a)\geq q.
	\end{split}
	\end{equation}
	Now if $L^+_{a}(q)$ is closed for some $q>0$ and across a sequence $a_j,j\geq 1,$ such that $\limi{j}a_j=\infty$ then $\mathbb{P}(I^x =\infty)\geq q>0$ and from the zero-one law for $I^x$ we conclude $\mathbb{P}(I^x=\infty)=1$. Otherwise, if $L^+_{a}(q)$ is closed  along a sequence of pairs $\curly{\lbrb{q_j,a_j}}$ such that $\limi{j}q_j=1$ and $\limi{j}a_j=\infty$ taking limit in \eqref{eq:inter} we directly see that $\mathbb{P}(I^x > \infty)=1$.
	This concludes the claim.
\end{proof}
We are now in a position to finish the proof of Theorem \ref{main_theorem}
\begin{proof}[Proof of Theorem \ref{main_theorem}]
	Lemma \ref{lem:inf} already gives a necessary condition for the almost sure finiteness of the random variable $I^x$.  Observe, that under the conditions of Theorem \ref{main_theorem} we know from Corollary \ref{rem:classf} as well as Remark \ref{rem:closed}  that the first condition of Lemma  \ref{lem:inf}  is satisfied.
	
	Let us now instead assume that  
	\begin{equation*}
	\int\limits_{0}^{\infty}f(x+\xi_s) \, ds=\infty\,\text{ almost surely.}
	\end{equation*}
	In order to derive a contradiction let us assume that there is a measurable set $E$ is such 
	\begin{equation*}
	\Rb\setminus E\,\,\text{is $\Pb^x$-transient and } \,\, \int\limits_{E}f(x+y)U(dy)<\infty. 
	\end{equation*}
	Then, these properties imply that
	\begin{equation*}
	\int_{0}^\infty (f\mathbb{I}_{E})\lbrb{x+\xi_s} ds=\infty
	\end{equation*}
	since the process spends only finite amount of time in $\Rb\setminus E$ and the function is either locally bounded by assumption or is implied to be such when it is continuous. This contradicts the fact that
	\begin{equation*}
	\begin{split}
	&\Ebb{\int_{0}^\infty (f\mathbb{I}_{E})\lbrb{x+\xi_s} ds}=\int\limits_{E}f(x+y)U(dy)<\infty.
	\end{split}
	\end{equation*}
	Therefore, we conclude that does not exists a set $E$ with the above properties.
\end{proof}
\begin{proof}[Proof of Corollary \ref{main_cor}]
 We emphasize that the form of $f_n,n\geq 1$, and the fact that the $\curly{f_n}_{n\geq 1}$ is non-decreasing show that Lemma \ref{lem:L} holds true, that is $L^+_a(q)$  ($q \in [0, 1)$) are closed,  which in turn gives the validity of the first item of Lemma \ref{lem:inf}.  Assume next that $\int\limits_{E}f(x+y)U(dy) = \infty\, \forall E\in\Bc\lbrb{\Rb}$ with  $\Rb\setminus E$ being $\Pb^x$-transient. Then, the second item of  Lemma \ref{lem:inf} is true and from it we get that $\mathbb{P}(I^x = \infty) = 1$. The fact that $\int\limits_{0}^{\infty}f(x+\xi_s)ds=\infty$ a.s. and $\int\limits_{E}f(x+y)U(dy)<\infty$ cannot hold simultaneously  for some 	$\Rb\setminus E$ being $\Pb^x$-transient is then proved the same way as in Theorem \ref{main_theorem} when $f$ is locally bounded. Finally, according to \cite[p. 152]{shilov} every non-negative function $f$ in the class $C^+$ therein can be approximated, away from a Lebesgue zero set, say $A_f$, by a non-decreasing sequence of step functions of the type defined in the statement of the current corollary. Also $C^+$ is the class of all Lebesque integrable non-negative functions, see \cite[p.156]{shilov}, and if $\Pbb{x+\xi_s\in A}=0,$ for all $s>0,x\in \Rb$ and for any set $A$ of zero Lebesgue measure the first part of this corollary is applicable and yields the final claim that Theorem \ref{main_theorem} is valid for any non-negative, measurable, locally-bounded function $f$. The fact the argument above is still valid for locally integrable function when $\int_{0}^{t}f(x+\xi_s)ds<\infty$ for all $t>0$ is immediate since we have used local boundedness only to ensure that integrals on finite time horizons are almost surely finite regardless of any other conditions.
\end{proof}
\section*{Acknowledgment}
The authors would like to thank our previous students Velislav Bodurov and Lukas Wiens, who worked in their Master theses through some of our previous preliminary versions of the results outlined above.\\  
The authors are also grateful to two anonymous referees for their careful reading which has led to the significant improvement of our work both as quality of exposition and mathematical results. \\
 The second author acknowledges the  support by the Grant
No BG05M2OP001-1.001-0003, financed by the Science and Education for Smart Growth
Operational Program (2014-2020) and co-financed by the European Union through the
European structural and Investment funds.


\begin{thebibliography}{}

\bibitem[\protect\citeauthoryear{Batty,~C.~J.~K.}{1992}]{batty}
 Batty,~C.~J.~K. (1992) Asymptotic stability of Schr\"{o}dinger semigroups: path integral methods, \textit{Math. Ann.},
 \textbf{292}, 457-492.
 \MR{1152946}

\bibitem[\protect\citeauthoryear{Bertoin,~J.}{1996}]{bertoin}
Bertoin,~J. (1996) 
\textit{L\'{e}vy Processes}  Cambridge Tracts in Mathematics 121, Cambridge University Press, Cambridge.


\bibitem[\protect\citeauthoryear{Bertoin,~J. and Savov,~M.}{2011}]{bertoinS}
Bertoin,~J. and Savov,~M. (2011) Some applications of duality for \LL processes in a half-line. \textit{Bull. Lond. Math. Soc.},
\textbf{43}, 97-110.
\MR{2765554}

\bibitem[\protect\citeauthoryear{D\"{o}ring,~L. and  Kyprianou, A.}{2015}]{doring}
D\"{o}ring,~L. and  Kyprianou, A. (2015) Perpetual Integrals for L\'{e}vy Processes. \textit{J. Theor. Probab.},
\textbf{29}, 1192-1198.
\MR{3540494}


\bibitem[\protect\citeauthoryear{Erickson, K. and Maller, R.}{2005}]{ErMa}
Erickson, K. and Maller, R. (2005) Generalised Ornstein-Uhlenbeck processes and the convergence of Lévy
integrals. \textit{Séminaire de Probabilités XXXVIII, Lecture Notes in Mathematics, 1857, Springer, Berlin,},
\textbf{29}, 70-94.
\MR{2126967}


\bibitem[\protect\citeauthoryear{Fitzsimmons, P.J. and Pitman, J.}{1999}]{FP}
Fitzsimmons, P.J. and Pitman, J. (1999)
Kac’s moment formula and the Feynman–Kac formula
for additive functionals of a Markov process. \textit{Stochastic Process. Appl.},
\textbf{79}, 117--134.
\MR{1670526}

\bibitem[\protect\citeauthoryear{Kallenberg.~O.}{2001}]{kallenberg}
Kallenberg.~O. (2001) 
\textit{Foundations of Modern Probability} Probability and its Applications, Springer-Verlag, New York.

\bibitem[\protect\citeauthoryear{Khoshnevisan, D.}{2002}]{Dhavar}
Khoshnevisan, D. (2002) 
\textit{Multiparameter Processes
	- An Introduction to Random Fields} pringer Monographs in Mathematics, Springer-Verlag, New York.
\bibitem[\protect\citeauthoryear{Kuhn, F.}{2019}]{K019}
Kuhn, F. (2019)
Perpetual integrals via random time changes. \textit{Bernoulli},
\textbf{25}, 1755--1769.
\MR{MR3961229}

\bibitem[\protect\citeauthoryear{Patie, P. and Savov, M.}{2018}]{PS18}
Patie, P. and Savov, M. (2018)
Bernstein-gamma functions and exponential functionals of Lévy processes \textit{Electron. J. Probab.},
\textbf{23} No.75, 1--101.
\MR{MR3835481}

\bibitem[\protect\citeauthoryear{Pinsky, R.}{2008}]{R08}
Pinsky, R. (2008)
A probabilistic approach to bounded/positive solutions for Schr\"odinger operators with certain
classes of potentials. \textit{Trans. Amer. Math. Soc.},
\textbf{360}, 6545--6554.
\MR{2434298}

\bibitem[\protect\citeauthoryear{Schilling, R. and Wang, J.}{2011}]{S11}
Schilling, R. and Wang, J. (2011)
On the coupling property of L\'vy processes. \textit{Annales de l'I.H.P. Probabilités et statistiques. },
\textbf{47}, No. 4, 1147--1159.
\MR{2884228}

\bibitem[\protect\citeauthoryear{Shilov.~G.~Ye.}{1965}]{shilov}
Shilov.~G.~Ye. (1965) 
\textit{Mathematical Analysis: A special course.} Pergamon Press.

\bibitem[\protect\citeauthoryear{Thorisson, H.}{2000}]{T00}
Thorisson, H. (2000) 
\textit{Coupling, Stationarity and Regeneration.} Springer.
\end{thebibliography}
\end{document}